\documentclass[11pt]{amsart}
\usepackage{amsmath,amsthm}
\usepackage{amstext,amsfonts,bm,amssymb}
\usepackage{graphics,graphicx}
\usepackage{a4wide}
\usepackage{float}
\usepackage{color}
\usepackage{verbatim}

\usepackage{gastex}

\theoremstyle{plain}
\newtheorem{theorem}{Theorem}

\newtheorem{Cor}{Corollary}
\newtheorem{Lem}{Lemma}

\def\fl#1{\left\lfloor#1\right\rfloor}

\def\stif#1#2{\left[#1\atop#2\right]} 
\def\sts#1#2{\left\{#1\atop#2\right\}}
\def\sttf2#1#2{\left[\!\!\left[#1\atop#2\right]\!\!\right]}  
\def\stf3f#1#2{\left[\!\!\left[\!\!\left[#1\atop#2\right]\!\!\right]\!\!\right]} 
\def\stff4#1#2{\left[\!\!\left[\!\!\left[\!\!\left[#1\atop#2\right]\!\!\right]\!\!\right]\!\!\right]}
\def\stss2#1#2{\left\{\!\!\left\{#1\atop#2\right\}\!\!\right\}}

\begin{document}

\title[A $q$-multiple zeta function at roots of unity]{Some explicit values of a $q$-multiple zeta function at roots of unity} 

\author{Takao Komatsu}
\address{Institute of Mathematics\\ Henan Academy of Sciences\\ Zhengzhou 450046\\ China;  \linebreak
Department of Mathematics, Institute of Science Tokyo, 2-12-1 Ookayama, Meguro-ku, Tokyo 152-8551, Japan}
\email{komatsu@hnas.ac.cn}

\date{
}

\begin{abstract}
In this paper, we give the values of a certain kind of $q$-multiple zeta functions at roots of unity. Various multiple zeta values have been proposed and studied by many researchers, but these multiple zeta values naturally arise from generalizations of Stirling numbers. It is interesting, but by no means easy, to show the values explicitly in certain cases. We give explicit formulas by using Bell polynomials, determinants, $r$-Stirling numbers, etc. The techniques used by Bachmann et al. for other multiple zeta values are also of help.  
\medskip

\end{abstract}

\subjclass[2010]{Primary 11M32; Secondary 05A15, 05A19, 05A30, 11B37, 11B73}
\keywords{multiple zeta functions, $q$-Stirling numbers with higher level, $r$-Stirling numbers, Bell polynomials, determinant}

\maketitle

\section{Introduction}\label{sec:1} 

Several different types of $q$-multiple zeta functions have been studied by many researchers (see, e.g., \cite{Bradley,OOZ,Zhao,Zudilin}). In \cite{Schlesinger} the function 
$$
\sum_{1\le i_1<i_2<\dots<i_m}\frac{1}{(1-q^{i_1})^{s_1}(1-q^{i_2})^{s_2}\cdots(1-q^{i_m})^{s_m}}
$$ 
is proposed. In \cite{Ko24}, in the course of studying generalizations of the classically famous Stirling numbers and their transformations, we naturally came to consider the following finite multiple zeta values. This is the case where the multiple zeta values mentioned above are finite and with $s=s_1=s_2=\dots=s_m$. Namely,  
\begin{equation}
\mathfrak Z_n(q;m,s):=\sum_{1\le i_1<i_2<\dots<i_m\le n-1}\frac{1}{(1-q^{i_1})^{s}(1-q^{i_2})^{s}\cdots(1-q^{i_m})^{s}}
\label{def:qmzv}
\end{equation}
is considered.   
By introducing $q$-generalized $(r,s)$-Stirling numbers $\sttf2{n}{k}_q^{(r,s)}$ and applying them to the values of the $q$-multiple zeta functions in (\ref{def:qmzv}), it is shown that 
\begin{equation}
\mathfrak Z_n(\zeta_n;m,1)=\frac{1}{m+1}\binom{n-1}{m} 
\label{eq:zzm1}
\end{equation} 
together with some more specific values.  Here, $\zeta_n=e^{2\pi\sqrt{-1}/n}$ is the $n$-th primitive root of unity.  
It is noticed that from the first identity of Lemma \ref{lem:exp11} below 
$$
\sum_{r\le i_1<i_2<\dots<i_m\le n-1}\frac{1}{(1-q^{i_1})^{s}(1-q^{i_2})^{s}\cdots(1-q^{i_m})^{s}}=\frac{([r-1]_q!)^s}{(1-q)^{s m}([n-1]_q!)^s}\sttf2{n}{m+1}_q^{(r,s)}\,. 
$$ 

In this paper, we shall show more values of $\mathfrak Z_n(\zeta_n;m,s)$ as well as their relations. 
More precisely, the values $\mathfrak Z_n(\zeta_n;m,s)$ are expressed in terms of $\mathfrak Z_n(\zeta_n;1,m s)$ by using the Bell polynomials (Theorem \ref{th:zz-det}) or the determinant (Theorem \ref{th:zz-det2}). Then using the inversion formula, we can conversely express $\mathfrak Z_n(\zeta_n;1,m s)$ in terms of a determinant in terms of $\mathfrak Z_n(\zeta_n;m,s)$. This fact allows us to express the value of $\mathfrak Z_n(\zeta_n;1,s)$ in a concise determinant expression (Corollary \ref{cor:zz-det2}). 

Another aim of this paper is to give the precise expression of $\mathfrak Z_n(\zeta_n;m,2)$ (Theorem \ref{th:s2}). Notice that the expression of $\mathfrak Z_n(\zeta_n;m,1)$ is given as (\ref{eq:zzm1}) in \cite{Ko24}. Though the identity looks simple, the proof is not as simple as all, but using the essential techniques by Bachmann Although the identity is deceptively simple, its proof is not simple at all but makes use of a substantial argument of Bachmann et al. in \cite{BTT18,BTT20,LP19,Tasaka21}. Surprisingly, the identity can also be expressed in terms of $r$-Stirling numbers introduced by Broder (Corollary \ref{cor:s2}).

\section{$q$-generalized $(r,s)$-Stirling numbers}  

Let $[n]_q$ denote the $q$-number, defined by 
$$
[n]_q=\frac{q^n-1}{q-1}\quad(q\ne 1)\,. 
$$  
Its $q$-factorial is given by $[n]_q!=[n]_q[n-1]_q\cdots[1]_q$ with $[0]_q!=1$. 
Let $r$ be a positive integer.    
The $q$-version of $r$-Stirling numbers of the first kind with higher level (level $s$) are denoted by $\sttf2{n}{k}_q^{(r,s)}$, and appear in the coefficients in the expansion 
\begin{equation}  
(x)_{n,q}^{(r,s)}=\sum_{k=0}^n(-1)^{n-k}\sttf2{n}{k}_q^{(r,s)}x^k\,,  
\label{def:qrsth1}  
\end{equation} 
where for $r,s\ge 1$, $(x)_{n,q}^{(r,s)}$ is defined by  
$$
(x)_{n,q}^{(r,s)}:=x^r\prod_{i=r}^{n-1}\bigl(x-([i]_q)^s\bigr)\quad(n>r)
$$ 
with $(x)_{r,q}^{(r,s)}=x^r$. 
When $r=s=1$, $s_q(n,k)=(-1)^{n-k}\sttf2{n}{k}_q^{(1,1)}$ are the signed $q$-Stirling numbers of the first kind (see, e.g., \cite{Ernst}), and $\sttf2{n}{k}_q=\sttf2{n}{k}_q^{(1,1)}$ are the unsigned $q$-Stirling numbers of the first kind. When $r=1$ and $q\to 1$, $\sttf2{n}{k}^{(s)}=\sttf2{n}{k}_1^{(1,s)}$ are the (unsigned) Stirling numbers of the first kind with higher level (\cite{KRV1}). When $r=s=1$ and $q\to 1$, $\stif{n}{k}=\sttf2{n}{k}_1^{(1,1)}$ are the unsigned Stirling numbers of the first kind.

The $q$-version of $r$-Stirling numbers of the second kind with higher level are denoted by $\stss2{n}{k}_q^{(r,s)}$, and appear in the coefficients in the expansion 
\begin{equation}  
x^n=\sum_{k=0}^n\stss2{n}{k}_q^{(r,s)}(x)_{k,q}^{(r,s)}\,.   
\label{def:qrsth2}  
\end{equation} 

When $r=1$ and $q\to 1$, $\stss2{n}{k}^{(s)}=\stss2{n}{k}_q^{(1,s)}$ are the Stirling numbers of the second kind with higher level, studied in \cite{KRV2}. 
When $r=s=1$, $S_q(n,k)=\stss2{n}{k}_q=(-1)^{n-k}\stss2{n}{k}_q^{(1,1)}$ are the signed $q$-Stirling numbers of the second kind (see, e.g., \cite{Ernst}). 
When $r=s=1$ and $q\to 1$, $\sts{n}{k}=\stss2{n}{k}_1^{(1,1)}$ are the classical Stirling numbers of the second kind.

From the definitions in \eqref{def:qrsth1} and \eqref{def:qrsth2}, the following orthogonal relations are yielded. 

\begin{Lem}  
\begin{align} 
\sum_{k=0}^{\max\{n,m\}}(-1)^{n-k}\sttf2{n}{k}_q^{(r,s)}\stss2{k}{m}_q^{(r,s)}&=\delta_{n,m}\,,
\label{ortho1}\\ 
\sum_{k=0}^{\max\{n,m\}}(-1)^{k-m}\stss2{n}{k}_q^{(r,s)}\sttf2{k}{m}_q^{(r,s)}&=\delta_{n,m}\,,
\label{ortho2} 
\end{align}
where $\delta_{n,m}$ is the Kronecker delta. 
\label{lem:ortho}
\end{Lem}

\subsection{First kind} 
  
The recurrence relation is given by 
\begin{equation}  
\sttf2{n}{k}_q^{(r,s)}=\sttf2{n-1}{k-1}_q^{(r,s)}+\bigl([n-1]_q\bigr)^s\sttf2{n-1}{k}_q^{(r,s)}
\label{rec:qsth1} 
\end{equation}
with 
\begin{align*}  
&\sttf2{n}{k}_q^{(r,s)}=0\quad(0\le k\le r,\, n\ge k),\\ 
&\sttf2{n}{k}_q^{(r,s)}=0\quad(n<k),\quad \sttf2{0}{0}_q^{(r,s)}=1\,. 
\end{align*}

From the recurrence relation \eqref{rec:qsth1}, we can see some initial values:  
\begin{align*}  
&\sttf2{n}{r}_q^{(r,s)}=\left(\frac{[n-1]_q!}{[r-1]_q!}\right)^s,\quad \sttf2{n}{r+1}_q^{(r,s)}=\left(\frac{[n-1]_q!}{[r-1]_q!}\right)^s\sum_{j=r}^{n-1}\frac{1}{\bigl([j]_q\bigr)^{s}}\,,\\
&\sttf2{n}{r+2}_q^{(r,s)}=\frac{1}{2}\left(\frac{[n-1]_q!}{[r-1]_q!}\right)^s\left(\left(\sum_{j=r}^{n-1}\frac{1}{\bigl([j]_q\bigr)^{s}}\right)^2-\sum_{j=r}^{n-1}\frac{1}{\bigl([j]_q\bigr)^{2 s}}\right)\\ 
&=\left(\frac{[n-1]_q!}{[r-1]_q!}\right)^s\sum_{r\le i<j\le n-1}\frac{1}{\bigl([i]_q[j]_q\bigr)^s}\,,\\
&\sttf2{n}{n}_q^{(r,s)}=1,\quad \sttf2{n}{n-1}_q^{(r,s)}=\sum_{j=r}^{n-1}\bigl([j]_q\bigr)^s,\quad \sttf2{n}{n-2}_q^{(r,s)}=\sum_{r\le i<j\le n-1}\bigl([i]_q[j]_q\bigr)^s\,.
\end{align*} 
\bigskip

In general, by using the recurrence relation (\ref{rec:qsth1}), we have expressions with combinatorial summations (\cite{Ko24}).  

\begin{Lem}  
For $r\le m\le n-1$ and $r\ge 1$, we have 
$$ 
\sttf2{n}{m}_q^{(r,s)}=\left(\frac{[n-1]_q!}{[r-1]_q!}\right)^s\sum_{r\le i_1<\dots<i_{m-r}\le n-1}\frac{1}{\bigl([i_1]_q\dots[i_{m-r}]_q\bigr)^s}\,.
$$ 
For $n-m\ge r$ and $r\ge 1$, we have 
\begin{align*}
\sttf2{n}{n-m}_q^{(r,s)}&=\sum_{r\le i_1<\dots<i_m\le n-1}\bigl([i_1]_q\dots[i_m]_q\bigr)^s\\
&=\sum_{r\le i_1\le i_2\le\dots\le i_m\le n-m}([i_1]_q[i_2+1]_q\cdots[i_m+m-1]_q)^s\\
&=\sum_{i_m=r}^{n-m}([i_m+m-1]_q)^s\sum_{i_{m-1}=r}^{i_m}([i_{m-1}+m-2]_q)^s\cdots\sum_{i_{2}=r}^{i_3}([i_{2}+1]_q)^s\sum_{i_{1}=r}^{i_2}([i_{1}]_q)^s\,.
\end{align*}
\label{lem:exp11} 
\end{Lem}

\subsection{Second kind}   

The recurrence relation is given by 
\begin{equation}  
\stss2{n}{k}_q^{(r,s)}=\stss2{n-1}{k-1}_q^{(r,s)}+\bigl([k]_q\bigr)^s\stss2{n-1}{k}_q^{(r,s)}
\label{rec:qsth2} 
\end{equation}
with 
\begin{align*}  
&\stss2{n}{k}_q^{(r,s)}=0\quad(0\le k\le r-1,~n\ge k),\quad \stss2{0}{0}_q^{(r,s)}=1,\\ 
&\stss2{n}{k}_q^{(r,s)}=0\quad(n\le k)\,. 
\end{align*}

From the recurrence relation \eqref{rec:qsth2}, we can see some initial values:  
\begin{align*}  
&\stss2{n}{r}_q^{(r,s)}=([r]_q)^{(n-r)s},\quad \stss2{n}{r+1}_q^{(r,s)}=\sum_{i=0}^{n-r-1}\bigl([r+1]_q\bigr)^{(n-r-i-1)s}([r]_q)^{i s}\,,\\
&\stss2{n}{r+2}_q^{(r,s)}=\sum_{j=0}^{n-r-2}\bigl([r+2]_q\bigr)^{(n-r-j-2)s}\sum_{i=0}^j\bigl([r+1]_q\bigr)^{(j-i)s}([r]_q)^{i s}\,,\\
&\stss2{n}{n}_q^{(r,s)}=1,\quad \stss2{n}{n-1}_q^{(r,s)}=\sum_{j=r}^{n-1}\bigl([j]_q\bigr)^s,\quad \stss2{n}{n-2}_q^{(r,s)}=\sum_{r\le i\le j\le n-2}\bigl([i]_q[j]_q\bigr)^s\,.
\end{align*} 

In general, the $q$-Stirling numbers of the second kind with higher level can be expressed in terms of iterated summations.  

\begin{Lem}  
For $r+1\le k\le n$ and $r\ge 1$,  
\begin{multline*}
\stss2{n}{k}_q^{(r,s)}=\sum_{i_{k-r}=0}^{n-k}\bigl([k]_q\bigr)^{(n-k-i_{k-r})s}\sum_{i_{k-r-1}=0}^{i_{k-r}}\bigl([k-1]_q\bigr)^{(i_{k-r}-i_{k-r-1})s}\\
\cdots\sum_{i_2=0}^{i_3}\bigl([r+2]_q\bigr)^{(i_{3}-i_{2})s}\sum_{i_1=0}^{i_2}\bigl([r+1]_q\bigr)^{(i_{2}-i_{1})s}\bigl([r]_q\bigr)^{i_{1} s}\,. 
\end{multline*}
For $n-k\ge r\ge 1$, 
$$
\stss2{n}{n-k}_q^{(r,s)}=\sum_{r\le i_1\le i_2\le\cdots\le i_k\le n-k}([i_1]_q[i_2]_q\cdots[i_k]_q)^s\,. 
$$ 
\label{lem:iterated-sum} 
\end{Lem}

\section{Main relations}

We shall use the following expression (see, e.g., \cite[p.247,(6e)]{Comtet}). 

\begin{Lem} 
Let $n$, $M$ and $K$ be positive integers with $M\ge K$. 
For $g_n:=a_1^n+a_2^n+\cdots+a_M^n$ we have 
$$
\sum_{1\le j_1<j_2<\dots<j_K\le M}a_{j_1}a_{j_2}\cdots a_{j_K}=\frac{1}{K!}\mathbf Y_n(g_1,-1!g_2,2! g_3,-3! g_4,\dots)\,, 
$$
where $\mathbf Y_n(x_1,x_2,\dots,x_n)$ is the (complete exponential) Bell polynomial, defined by 
$$
\exp\left(\sum_{m=1}^\infty x_m\frac{t^m}{m!}\right)=1+\sum_{n=1}^\infty\mathbf Y_n(x_1,x_2,\dots,x_n)\frac{t^n}{n!}\,. 
$$ 
That is, 
\begin{align*}
&\mathbf Y_n(x_1,x_2,x_3,\dots,x_n)\\ 
&=\sum_{k=1}^n\sum_{i_1+2 i_2+\cdots+(n-k+1)i_{n-k+1}=n\atop i_1+i_2+i_3+\cdots=k}\frac{n!}{i_1!i_2!\cdots i_{n-k+1}!}\left(\frac{x_1}{1!}\right)^{i_1}\left(\frac{x_2}{2!}\right)^{i_2}\cdots\left(\frac{x_{n-k+1}}{(n-k+1)!}\right)^{i_{n-k+1}}
\end{align*} 
with $\mathbf Y_0=1$. 
\label{lem:bell}
\end{Lem}

Hence, if $g_\nu$ is replaced by $\mathfrak Z_n(\zeta_n;1,\nu)$, then we have
\begin{align*}
&\mathfrak Z_n(\zeta_n;m,1)=\frac{1}{m+1}\binom{n-1}{m}\\
&=\sum_{i_1+2 i_2+\cdots=m}\frac{1}{i_1!i_2!\cdots}\left(\frac{\mathfrak Z_n(\zeta_n;1,1)}{1}\right)^{i_1}\left(-\frac{\mathfrak Z_n(\zeta_n;1,2)}{2}\right)^{i_2}\\
&\qquad\qquad\qquad\cdot\left(\frac{\mathfrak Z_n(\zeta_n;1,3)}{3}\right)^{i_3}\left(-\frac{\mathfrak Z_n(\zeta_n;1,4)}{4}\right)^{i_4}\cdots\,.
\end{align*}

In general, we have 
\begin{align*}
&\mathfrak Z_n(\zeta_n;m,s)\\
&=\sum_{i_1+2 i_2+\cdots=m}\frac{1}{i_1!i_2!\cdots}\left(\frac{\mathfrak Z_n(\zeta_n;1,s)}{1}\right)^{i_1}\left(-\frac{\mathfrak Z_n(\zeta_n;1,2 s)}{2}\right)^{i_2}\\
&\qquad\qquad\qquad\cdot\left(\frac{\mathfrak Z_n(\zeta_n;1,3 s)}{3}\right)^{i_3}\left(-\frac{\mathfrak Z_n(\zeta_n;1,4 s)}{4}\right)^{i_4}\cdots\,.
\end{align*}

By using the Bell polynomials, we have 

\begin{theorem}
\begin{align*}
&\mathfrak Z_n(\zeta_n;m,s)\\
&=\frac{1}{m!}\mathbf Y_m\bigl(\mathfrak Z_n(\zeta_n;1,s),-1!\mathfrak Z_n(\zeta_n;1,2 s),2!\mathfrak Z_n(\zeta_n;1,3 s),-3!\mathfrak Z_n(\zeta_n;1,4 s),\dots)\bigr)\,.
\end{align*}
\label{th:zz-bell}
\end{theorem} 

By using the equivalent relations in Lemma \ref{lem:gtrudi} below, we also have the following determinant expression.

\begin{theorem} 
For integers $n,m$ with $n\ge 2$ and $m\ge 1$, we have 
\begin{align*}
&\mathfrak Z_n(\zeta_n;m,s)\\
&=\frac{1}{m!}\left|\begin{array}{ccccc}
\mathfrak Z_n(\zeta_n;1,s)&1&0&\cdots&\\ 
\mathfrak Z_n(\zeta_n;1,2 s)&\mathfrak Z_n(\zeta_n;1,s)&2&&\vdots\\ 
\vdots&&\ddots&&0\\
\mathfrak Z_n\bigl(\zeta_n;1,(m-1)s\bigr)&\mathfrak Z_n\bigl(\zeta_n;1,(m-2)s\bigr)&\cdots&\mathfrak Z_n(\zeta_n;1,s)&m-1\\ 
\mathfrak Z_n(\zeta_n;1,m s)&\mathfrak Z_n\bigl(\zeta_n;1,(m-1)s\bigr)&\cdots&\mathfrak Z_n(\zeta_n;1,2 s)&\mathfrak Z_n(\zeta_n;1,s)\\ 
\end{array}
\right|\,.
\end{align*}
\label{th:zz-det}
\end{theorem}

\begin{theorem}  
For integers $n,m$ with $n\ge 2$ and $m\ge 1$, we have 
\begin{align*}
&\mathfrak Z_n(\zeta_n;1,m s)\\
&=\left|\begin{array}{ccccc}
\mathfrak Z_n(\zeta_n;1,s)&1&0&\cdots&\\ 
2\mathfrak Z_n(\zeta_n;2,s)&\mathfrak Z_n(\zeta_n;1,s)&1&&\vdots\\ 
\vdots&&\ddots&&0\\
(m-1)\mathfrak Z_n(\zeta_n;m-1,s)&\mathfrak Z_n(\zeta_n;m-2,s)&\cdots&\mathfrak Z_n(\zeta_n;1,s)&1\\ 
m\mathfrak Z_n(\zeta_n;m,s)&\mathfrak Z_n(\zeta_n;m-1,s)&\cdots&\mathfrak Z_n(\zeta_n;2,s)&\mathfrak Z_n(\zeta_n;1,s)\\ 
\end{array}
\right|\,.
\end{align*}
\label{th:zz-det2}
\end{theorem}  

If we put $s=1$ and $m$ is replaced by $s$ in Theorem \ref{th:zz-det2}, by using (\ref{eq:zzm1}) we have the determinat formula for $\mathfrak Z_n(\zeta_n;1,s)$.  

\begin{Cor}  
For integers $n,s$ with $n,s\ge 2$, we have 
$$
\mathfrak Z_n(\zeta_n;1,s)
=\left|\begin{array}{ccccc}
\frac{n-1}{2}&1&0&\cdots&\\ 
\frac{2}{3}\binom{n-1}{2}&\frac{n-1}{2}&1&&\vdots\\ 
\vdots&&\ddots&&0\\
\frac{s-1}{s}\binom{n-1}{s-1}&\frac{1}{s-1}\binom{n-1}{s-2}&\cdots&\frac{n-1}{2}&1\\ 
\frac{s}{s+1}\binom{n-1}{s}&\frac{1}{s}\binom{n-1}{s-1}&\cdots&\frac{1}{3}\binom{n-1}{2}&\frac{n-1}{2}\\ 
\end{array}
\right|\,.
$$
\label{cor:zz-det2}
\end{Cor}

By taking $s=2,3,\dots,9$ in Corollary \ref{cor:zz-det2}, we have 
\begin{align*}
\mathfrak Z_n(\zeta_n;1,2)&=-\frac{(n-1)(n-5)}{12}\,,\\
\mathfrak Z_n(\zeta_n;1,3)&=-\frac{(n-1)(n-3)}{8}\,,\\
\mathfrak Z_n(\zeta_n;1,4)&=\frac{(n-1)(n^3+n^2-109 n+251)}{6!}\,,\\
\mathfrak Z_n(\zeta_n;1,5)&=\frac{(n-1)(n-5)(n^2+6 n-19)}{288}\,,\\
\mathfrak Z_n(\zeta_n;1,6)&=-\frac{(n-1)(2 n^5+2 n^4-355 n^3-355 n^2+11153 n-19087)}{12\cdot 7!}\,,\\
\mathfrak Z_n(\zeta_n;1,7)&=-\frac{(n-1)(n-7)(2 n^4+16 n^3-33 n^2-376 n+751)}{24\cdot 6!}\,,\\
\mathfrak Z_n(\zeta_n;1,8)&=\frac{(n-1)(3 n^7+3 n^6-917 n^5-917 n^4+39697 n^3+39697 n^2-744383 n+1070017)}{10!}\,,\\
\mathfrak Z_n(\zeta_n;1,9)&=\frac{27(n-1)(n-3)(n-9)(n^5+13 n^4+10 n^3-350 n^2-851 n+2857)}{2\cdot 10!}\,.
\end{align*} 
The constant terms of $\mathfrak Z_n(\zeta_n;1,s)$ ($s=1,2,\dots$) are given by 
$$
-\frac{1}{2},-\frac{5}{12},-\frac{3}{8},-\frac{251}{720},-\frac{95}{288},-\frac{19087}{60480},-\frac{5257}{17280},\dots\,,
$$
which are equal to 
$$
(-1)^{s-1}\frac{B_{s}^{(s)}}{s!}\quad(s=1,2,\dots)\,.  
$$
Here, N\"orlund numbers $B_{n}^{(n)}$ are given by 
$$
\frac{t}{(1+t)\log(1+t)}=\sum_{n=0}^\infty B_n^{(n)}\frac{t^n}{n!}
$$
or the Bernoulli numbers of order $\alpha$ by 
$$
\left(\frac{t}{e^t-1}\right)^\alpha=\sum_{n=0}^\infty B_n^{(\alpha)}\frac{t^n}{n!}\,.
$$

\subsection{Proof of the results} 

For two sequences $a_0=1,a_1,a_2,\dots$ and $b_0=1,b_1,b_2,\dots$ we have the equivalent expressions. 

\begin{Lem}  
The following expressions are equivalent.  
\begin{enumerate}
\item[(1)] $\displaystyle b_m=\sum_{i_1+2 i_2+\cdots+m i_m=m\atop i_1,i_2,\dots,i_m\ge 0}\frac{1}{i_1!i_2!\cdots i_m!}\left(\frac{a_1}{1}\right)^{i_1}\left(\frac{-a_2}{2}\right)^{i_2}\cdots\left(\frac{(-1)^{m-1}a_m}{m}\right)^{i_m}$
\item[(2)] $\displaystyle b_m=\frac{1}{m!}\left|\begin{array}{ccccc}
a_1&1&0&\cdots&0\\
a_2&a_1&2&&\vdots\\
\vdots&&\ddots&&0\\
a_{m-1}&a_{m-2}&\cdots&a_1&m-1\\
a_m&a_{m-1}&\cdots&a_2&a_1\\
\end{array}
\right|$ 
\item[(3)] $\displaystyle a_n=\left|\begin{array}{ccccc}
b_1&1&0&\cdots&0\\
2 b_2&b_1&1&&\vdots\\
\vdots&&\ddots&&0\\
(n-1)b_{n-1}&b_{n-2}&\cdots&b_1&1\\
n b_n&b_{n-1}&\cdots&b_2&b_1\\
\end{array}
\right|$ 
\item[(4)] $\displaystyle m b_m=\sum_{i=1}^m(-1)^{i-1}a_i b_{m-i}$
\item[(5)] $\displaystyle a_n=\sum_{j=1}^{n-1}(-1)^{j-1}b_j a_{n-j}+(-1)^{n+1}n b_n$ 
\end{enumerate}
\label{lem:gtrudi}
\end{Lem} 
\begin{proof}
(1) and (2) are equivalent because of a simple modification of Trudi's formula (\cite[Vol.3, p.214]{Muir},\cite{Trudi}): 
\begin{multline*} 
\left|
\begin{array}{ccccc}
a_1  & a_0  &  \cdots   & & 0  \\
a_2  & a_{1}    &  \cdots  & &   \\
\vdots  &  \vdots &  \ddots  &  & \vdots  \\
a_{n-1}  &a_{n-2}    &  \cdots  &a_1  & a_0  \\
a_n  & a_{n-1}   &  \cdots  & a_2  & a_1
\end{array}
\right|\\
=
\sum_{i_1 + 2 i_2 + \cdots +n t_n=n}\frac{(i_1+\cdots + i_n)!}{i_1 !\cdots i_n !}(-a_0)^{n-i_1-\cdots - i_n}a_1^{i_1}a_2^{i_2}\cdots a_n^{i_n}\,.  
\end{multline*}
By expanding the determinant of (2) in the last column, we can see that (2) and (4) are equivalent. Similarly, by expanding the determinant of (3) in the last column, we can see that (3) and (5) are equivalent. 
By applying the inversion formula  (see, e.g. \cite[Lemma 1]{Ko24},\cite[Theorem 1]{KR}), we can find that (4) and (5) are equivalent.  
\end{proof}

\begin{proof}[Proof of Theorem \ref{th:zz-det}.]
By Lemma \ref{th:zz-det}, we have 
{\small  
\begin{align*}
&\mathfrak Z_n(\zeta_n;m,s)\\
&=\frac{1}{m!}\left|\begin{array}{ccccc}
\mathfrak Z_n(\zeta_n;1,s)&-1&0&\cdots&\\ 
-\mathfrak Z_n(\zeta_n;1,2 s)&\mathfrak Z_n(\zeta_n;1,s)&-2&&\vdots\\ 
\vdots&&\ddots&&0\\
(-1)^{m-2}\mathfrak Z_n\bigl(\zeta_n;1,(m-1)s\bigr)&(-1)^{m-3}\mathfrak Z_n\bigl(\zeta_n;1,(m-2)s\bigr)&\cdots&\mathfrak Z_n(\zeta_n;1,s)&-m+1\\ 
(-1)^{m-1}\mathfrak Z_n(\zeta_n;1,m s)&(-1)^{m-2}\mathfrak Z_n\bigl(\zeta_n;1,(m-1)s\bigr)&\cdots&-\mathfrak Z_n(\zeta_n;1,2 s)&\mathfrak Z_n(\zeta_n;1,s)\\ 
\end{array}
\right|\\ 
&=\frac{1}{m!}\left|\begin{array}{ccccc}
\mathfrak Z_n(\zeta_n;1,s)&1&0&\cdots&\\ 
\mathfrak Z_n(\zeta_n;1,2 s)&\mathfrak Z_n(\zeta_n;1,s)&2&&\vdots\\ 
\vdots&&\ddots&&0\\
\mathfrak Z_n\bigl(\zeta_n;1,(m-1)s\bigr)&\mathfrak Z_n\bigl(\zeta_n;1,(m-2)s\bigr)&\cdots&\mathfrak Z_n(\zeta_n;1,s)&m-1\\ 
\mathfrak Z_n(\zeta_n;1,m s)&\mathfrak Z_n\bigl(\zeta_n;1,(m-1)s\bigr)&\cdots&\mathfrak Z_n(\zeta_n;1,2 s)&\mathfrak Z_n(\zeta_n;1,s)\\ 
\end{array}
\right|\,.
\end{align*} 
} 
\end{proof}

\subsection{Relations with another finite multiple zeta function}  

Different finite $q$-multiple zeta functions have been introduced and studied.  
In \cite{BTT18} the finite multiple harmonic $q$-series  
\begin{equation} 
z_n(q;s_1,\dots,s_m):=\sum_{n-1\ge i_1>\cdots>i_m\ge 1}\frac{q^{(s_1-1)i_1}\cdots q^{(s_m-1)i_m}}{([i_1]_q)^{s_1}\cdots([i_m]_q)^{s_m1}}
\label{def:btt} 
\end{equation} 
is studied (see also \cite{BTT20,HI17,Tasaka21}).  
Applying the results from this paper to our function in (\ref{def:qmzv}) leads to the expression of the value of $\mathfrak Z_n(q;1,s)$ in terms of the degenerate Bernoulli numbers $\beta_j(\lambda)$. The degenerate Bernoulli numbers are defined by Carlitz \cite{Carlitz56} as 
$$
\frac{t}{(1+\lambda t)^{1/\lambda}-1}=\sum_{k=0}^\infty\beta_k(\lambda)\frac{t^k}{k!}\,. 
$$ 
The initial values are given by 
$$
\{\beta_k(\lambda)\}_{k\ge 0}=1,\frac{\lambda-1}{2},-\frac{\lambda^2-1}{6},\frac{\lambda^3-\lambda}{4},-\frac{19\lambda^4-20\lambda^2+1}{30},\frac{3\lambda^5-10\lambda^3+\lambda}{4},\dots,
$$

\begin{theorem} 
$$ 
\mathfrak Z_n(q;1,s)=-\sum_{j=1}^s\binom{s-1}{j-1}\beta_j(n^{-1})\frac{n^j}{j!}\,. 
$$ 
\label{th:dgber}
\end{theorem}

\begin{proof}  
The function $\mathfrak Z_n(q;1,s)$ can be written in terms of $z_n(q;s_1,\dots,s_m)$ as $m=1$.  
$$ 
\mathfrak Z_n(q;1,s)=\sum_{j=1}^s\binom{s-1}{j-1}\frac{z_n(q;j)}{1-q}\,. 
$$ 
\label{rel:qmzv-btt} 
Since 
$$
\frac{z_n(\zeta_n;j)}{\bigl(n(1-\zeta_n)\bigr)^j}=-\frac{\beta_j(n^{-1})}{j!}
$$
\cite[(2.6)]{BTT18}, 
we have the desired result. 
\end{proof}


\section{The case $s=2$} 

In this section, we shall show an explicit formula for $\mathfrak Z_n(\zeta_n;m,2)$.   
Notice that the values $z_n(\zeta_n;\underbrace{s,\dots,s}_m)$ in (\ref{def:btt}) are explicitly given (see \cite[Theorem 1.1]{BTT20} for $m=1,2,3$).   

\begin{theorem} 
For a positive integer $m$, we have 
$$
\mathfrak Z_n(\zeta_n;m,2)=\frac{1}{n(m+1)}\left(\binom{n-1}{m}+(-1)^m\binom{n-1}{2 m+1}\right)\,. 
$$ 
\label{th:s2}
\end{theorem}

\noindent 
{\it Remark.}   
For $m=1,2,3,4$ we have 
\begin{align*}
\mathfrak Z_n(\zeta_n;1,2)&=-\frac{2(n-1)(n-5)}{4!}\,,\\
\mathfrak Z_n(\zeta_n;2,2)&=\frac{2(n-1)(n-2)(n^2-12 n+47)}{6!}\,,\\
\mathfrak Z_n(\zeta_n;3,2)&=-\frac{2(n-1)(n-2)(n-3)(n^3-22 n^2+179 n-638)}{8!}\,,\\
\mathfrak Z_n(\zeta_n;4,2)&=\frac{2(n-1)(n-2)(n-3)(n-4)(n^4-35 n^3+485 n^2-3325 n+11274)}{10!}\,,\\
\end{align*} 
It is interesting to see that the coefficients of the higher order polynomial factors, $n-5$, $n^2-12 n+47$, $n^3-22 n^2+179 n-638$, $n^4-35 n^3+485 n^2-3325 n+11274$, $\dots$, are included in the sequences \cite[A177938,A054655]{oeis}, which can be written in terms of the $r$-Stirling numbers of the first kind (see Corollary \ref{cor:s2}).

\begin{proof} 
By (\ref{def:qmzv}), using $\prod_{j=1}^{n-1}(1-\zeta_n^j)=n$, we have 
\begin{align*}
\sum_{m=0}^{n-1}\mathfrak Z_n(\zeta_n;m,s)X^m&=\sum_{m=0}^{n-1}\sum_{1\le i_1<i_2<\dots<i_m\le n-1}\frac{1}{(1-\zeta_n^{i_1})^{s}(1-\zeta_n^{i_2})^{s}\cdots(1-\zeta_n^{i_m})^{s}}X^m\\
&=\prod_{j=1}^{n-1}\left(1+\frac{X}{(1-\zeta_n^j)^s}\right)=\frac{1}{n^s}\prod_{j=1}^{n-1}\bigl((1-\zeta_n^j)^s+X\bigr)\,. 
\end{align*}
Here, we assume that $\mathfrak Z_n(\zeta_n;0,s)=1$. 
Let $\alpha_i$ ($i=1,2,\dots,s$) be the roots of the polynomial $(1-Y)^s+X$. Namely, 
\begin{equation}
\prod_{i=1}^s(\alpha_i-Y)=(1-Y)^s+X\,. 
\label{eq:alpha}
\end{equation}
For the elementary symmetric polynomials $e_j(x_1,\dots,x_s)$, defined by 
$$
\sum_{j=0}^s(-1)^j e_j(x_1,\dots,x_s)T^{s-j}:=\prod_{j=1}^s(T-x_j)\,, 
$$ 
we see that 
\begin{equation} 
e_j(\alpha_1,\dots,\alpha_s)=\binom{s}{j}+\delta_{j,s}X\,. 
\label{eq:elemensp}
\end{equation}
Since by (\ref{eq:alpha}) 
$$
\prod_{i=1}^s(\alpha_i-1)=X\,,  
$$ 
we get 
$$
\prod_{j=1}^{n-1}\prod_{i=1}^s(\alpha_i-\zeta_n^j)=\prod_{i=1}^s\frac{\alpha_i^n-1}{\alpha_i-1}=\frac{1}{X}\prod_{i=1}^s(\alpha_i^n-1)\,. 
$$
Thus, we have 
\begin{align*}
&\sum_{n=1}^\infty n^{s-1}Y^n\sum_{m=0}^{n-1}\mathfrak Z_n(\zeta_n;m,s)X^m=\sum_{n=1}^\infty n^{s-1}Y^n\frac{1}{n^s}\prod_{j=1}^{n-1}\bigl(1-\zeta_n^j)^s+X\bigr)\\
&=\sum_{n=1}^\infty\frac{Y^n}{n}\prod_{j=1}^{n-1}\prod_{i=1}^s(\alpha_i-\zeta_n^j)=\frac{1}{X}\sum_{n=1}^\infty\frac{Y^n}{n}\prod_{i=1}^s(\alpha_i^n-1)\\
&=\frac{1}{X}\sum_{n=1}^\infty\frac{Y^n}{n}\left((-1)^s+\sum_{l=1}^s(-1)^{s-l}\sum_{1\le i_1<\dots<i_l\le s}(\alpha_{i_1}\cdots\alpha_{i_l})^n\right)\\
&=\frac{(-1)^{s-1}}{X}\left(-\sum_{n=1}^\infty\frac{Y^n}{n}+\sum_{l=1}^s(-1)^{l-1}\sum_{1\le i_1<\dots<i_l\le s}\sum_{n=1}^\infty\frac{(\alpha_{i_1}\cdots\alpha_{i_l}Y)^n}{n}\right)\\
&=\frac{(-1)^{s-1}}{X}\left(\log(1-Y)+\sum_{l=1}^s(-1)^{l}\sum_{1\le i_1<\dots<i_l\le s}\log(1-\alpha_{i_1}\cdots\alpha_{i_l}Y)\right)\,. 
\end{align*}
Using the polynomials $F_{s,l}(X,Y)$, defined by 
\begin{align*}
F_{s,0}(X,Y)&=1-Y\quad(s\ge 1)\,,\\
F_{s,l}(X,Y)&=\prod_{1\le i_1<\cdots<i_l\le s}(1-\alpha_{i_1}\cdots\alpha_{i_l}Y)\quad(1\le l\le s)\,,
\end{align*}
we have 
\begin{equation}  
\sum_{n=1}^\infty n^{s-1}Y^n\sum_{m=0}^{n-1}\mathfrak Z_n(\zeta_n;m,s)X^m=\frac{(-1)^{s-1}}{X}\log\left(\prod_{l=0}^s F_{s,l}(X,Y)^{(-1)^l}\right)\,. 
\label{eq:log-f}
\end{equation}
The coefficients of $Y^m$ in $F_{s,l}(X,Y)$ can be written in terms of the elementary symmetric polynomials $e_j(\alpha_1,\dots,\alpha_s)$ in (\ref{eq:elemensp}). 
When $s=2$, by (\ref{eq:elemensp}) we get $\alpha_1+\alpha_2=2$ and $\alpha_1\alpha_2=X+1$. Hence, 
$$
F_{2,0}=1-Y,\quad F_{2,1}=(1-Y)^2+X Y^2,\quad F_{2,2}=1-(X+1)Y\,. 
$$
Substituting these values into (\ref{eq:log-f}), we have 
\begin{align*}
&\sum_{n=1}^\infty n Y^n\sum_{m=0}^{n-1}\mathfrak Z_n(\zeta_n;m,2)X^m=\sum_{m=0}^\infty X^m\sum_{n=m+1}^\infty n\mathfrak Z_n(\zeta_n;m,2)Y^n\\
&=-\frac{1}{X}\log\left(\prod_{l=0}^2 F_{2,l}(X,Y)^{(-1)^l}\right)\\
&=-\frac{1}{X}\log\left(\frac{(1-Y)\bigl(1-(X+1)Y\bigr)}{(1-Y)^2+X Y^2}\right)\\
&=-\frac{1}{X}\left(\log\left(1-\frac{X Y}{1-Y}\right)-\log\left(1+\frac{X Y^2}{(1-Y)^2}\right)\right)\\
&=\frac{1}{X}\left(\sum_{m=1}^\infty\frac{1}{m}\left(\frac{X Y}{1-Y}\right)^m+\sum_{m=1}^\infty\frac{(-1)^{m-1}}{m}\left(\frac{X Y^2}{(1-Y)^2}\right)^m\right)\\
&=\sum_{m=1}^\infty\frac{X^{m-1}Y^m}{m}\sum_{j=0}^\infty\binom{m+j-1}{m-1}Y^j+\sum_{m=1}^\infty\frac{(-1)^{m-1}X^{m-1}Y^{2 m}}{m}\sum_{j=0}^\infty\binom{2 m+j-1}{2 m-1}Y^j\\
&=\sum_{m=0}^\infty\frac{X^{m}}{m+1}\sum_{j=0}^\infty\binom{m+j}{m}Y^{m+j+1}+\sum_{m=0}^\infty\frac{(-1)^{m}X^{m}}{m+1}\sum_{j=0}^\infty\binom{2 m+j+1}{2 m+1}Y^{2 m+j+2}\\
&=\sum_{m=0}^\infty\frac{X^{m}}{m+1}\sum_{n=m+1}^\infty\binom{n-1}{m}Y^{n}+\sum_{m=0}^\infty\frac{(-1)^{m}X^{m}}{m+1}\sum_{n=2 m+2}^\infty\binom{n-1}{2 m+1}Y^{n}\,. 
\end{align*}
Comparing the coefficients on both sides, we obtain that 
$$
\mathfrak Z_n(\zeta_n;m,2)=\frac{1}{n(m+1)}\left(\binom{n-1}{m}+(-1)^m\binom{n-1}{2 m+1}\right)\,. 
$$ 
\end{proof}

\subsection{A different expression}  

The expression of $\mathfrak Z_n(\zeta_n;m,2)$ in Theorem \ref{th:s2} can be written in terms of the $r$-Stirling numbers of the first kind,  
 introduced by Broder \cite{Broder} as   
\begin{align}
x^r(x-r)(x-r-1)\cdots(x-n+1)&=\sum_{k=0}^n(-1)^{n-k}\sttf2{n}{k}_1^{(r,1)}x^k\notag\\
&=\sum_{k=r}^n(-1)^{n-k}\sttf2{n}{k}_1^{(r,1)}x^k\,.
\label{def:rst1}
\end{align}  
That is the special case where $s=1$ and $q\to 1$ in (\ref{def:qrsth1}). 

By Lemma \ref{lem:exp11}, we have 
\begin{align*} 
\sttf2{2 m+2}{m+k+2}_1^{(m+1,1)}&=\sum_{m+1\le i_1<\dots<i_{m-k}\le 2 m+1}i_1\cdots i_{m-k}\\
&=\frac{(2 m+1)!}{m!}\sum_{m+1\le i_1<\dots<i_{k+1}\le 2 m+1}\frac{1}{i_1\cdots i_{k+1}}\,.
\end{align*}
In particular, we get 
$$
\sttf2{2 m+2}{m+2}_1^{(m+1,1)}=\frac{(2 m+1)!}{m!}(H_{2 m+1}-H_{m})=(m+1)h_{m+1}^{(m+1)}\,, 
$$ 
where $h_n^{(k)}$ is the hyperharmonic function, defined by 
$$
h_n^{(k)}=\sum_{i=1}^n h_{i}^{(k-1)}\quad(k\ge 2)\quad\hbox{with}\quad h_n^{(1)}=H_n 
$$
and $H_n=\sum_{i=1}^n 1/i$ is the $n$-th harmonic number. 

By Theorem \ref{th:s2}, we can show the following.  

\begin{Cor} 
For a positive integer $m$, we have 
$$
\mathfrak Z_n(\zeta_n;m,2)=\frac{2\cdot m!}{(2 m+2)!}\binom{n-1}{m}\sum_{k=0}^m\sttf2{2 m+2}{m+k+2}_1^{(m+1,1)}(-n)^k\,. 
$$ 
\label{cor:s2}
\end{Cor}

\noindent 
{\it Remark.}  
Note that  the coefficients of the higher order polynomial factors correspond to the sum of the $r$-Stirling numbers of the first kind.    
For example, when $m=4$, we can see that  
\begin{multline*}
n^4-35 n^3+485 n^2-3325 n+11274\\
=\sttf2{10}{10}_1^{(5,1)}n^4-\sttf2{10}{9}_1^{(5,1)} n^3+\sttf2{10}{8}_1^{(5,1)}n^2-\sttf2{10}{7}_1^{(5,1)}n+\sttf2{10}{6}_1^{(5,1)}\,.
\end{multline*} 
From the representation of the $r$-Stirling numbers of the first kind, we can also have 
$$ 
\mathfrak Z_n(\zeta_n;m,2)=\frac{1}{n}\binom{n}{m+1}\sum_{k=0}^m\sum_{m+1\le i_1<\dots<i_{k+1}\le 2 m+1}\frac{(-n)^k}{i_1\cdots i_{k+1}}\,.
$$

\begin{proof}  
By Theorem \ref{th:s2}, from the definition of the $r$-Stirling numbers of the first kind, we have 
\begin{align*}
&\mathfrak Z_n(\zeta_n;m,2)\\
&=\frac{(n-1)(n-2)\cdots(n-m)}{n(m+1)(2 m+1)!}\left(\underbrace{(m+1)(m+2)\cdots(2 m+1)}_{m+1}\right.\\
&\qquad\qquad\qquad\left.+(-1)^m\underbrace{(n-m-1)(n-m-2)\cdots(n-2 m-1)}_{m+1}\right)\\
&=\frac{2\cdot m!}{n(2 m+2)!}\binom{n-1}{m}(-1)^m\sum_{k=m+2}^{2 m+2}(-1)^{m+k}\sttf2{2 m+2}{k}_1^{(m+1,1)}n^{k-m-1}\\
&=\frac{2\cdot m!}{(2 m+2)!}\binom{n-1}{m}\sum_{k=0}^m\sttf2{2 m+2}{m+k+2}_1^{(m+1,1)}(-n)^k\,. 
\end{align*}
\end{proof}

\section{The case $s=3$} 

By a similar technique as in the proof of Theorem \ref{th:s2}, we can also obtain an explicit formula of $\mathfrak Z_n(\zeta_n;m,3)$.   

\begin{theorem} 
For a positive integer $m$, we have 
\begin{align*}
&\mathfrak Z_n(\zeta_n;m,3)\\
&=\frac{1}{n^2(m+1)}\left(\binom{n-1}{m}+\binom{n-1}{3 m+2}\right)\\
&\quad -\frac{1}{n^2}\sum_{k=0}^{\fl{\frac{m+1}{2}}}\sum_{i=0}^{m-2 k+1}\frac{1}{m-k+1}\binom{m-k+1}{k}\binom{m-2 k+1}{i}\binom{n+m-2 k-i}{3 m-3 k+2}2^i(-3)^{m-2 k-i+1}\,. 
\end{align*} 
\label{th:s3}
\end{theorem}

\noindent 
{\it Remark.}  
For $m=1,2,3,4$, we have 
\begin{align*}
\mathfrak Z_n(\zeta_n;1,3)&=-\frac{(n-1)(n-3)}{8}\,,\\ 
\mathfrak Z_n(\zeta_n;2,3)&=\frac{6(n-1)(n-2)(n^4+3 n^3+301 n^2-2883 n+6898)}{9!}\,,\\ 
\mathfrak Z_n(\zeta_n;3,3)&=-\frac{3(n-1)(n-2)(n-3)(n^5-4 n^4+100 n^3-2290 n^2+15019 n-32986)}{10!}\,,\\ 
\mathfrak Z_n(\zeta_n;4,3)&=\frac{2(n-1)(n-2)(n-3)(n-4)}{5\cdot 14!}(n^8+10 n^7+3705 n^6-53340 n^5\\
&\quad +360423 n^4-7406910 n^3+99197195 n^2-551374960 n+1157817876)\,. 
\end{align*}

\begin{proof}  
When $s=3$, by (\ref{eq:elemensp}) we get $\alpha_1+\alpha_2+\alpha_3=3$, $\alpha_1\alpha_2+\alpha_1\alpha_3+\alpha_2\alpha_3=3$ and $\alpha_1\alpha_2\alpha_3=X+1$. Hence, 
\begin{align*}
&F_{3,0}=1-Y,\quad F_{3,1}=(1-Y)^3-X Y^3,\\
&F_{3,2}=(1-Y)^3-(X Y+2 Y-3)X Y^2,\quad F_{3,3}=(1-Y)-X Y\,. 
\end{align*}
Substituting these values into (\ref{eq:log-f}), we have 
\begin{align*}
&\sum_{n=1}^\infty n^2 Y^n\sum_{m=0}^{n-1}\mathfrak Z_n(\zeta_n;m,3)X^m=\sum_{m=0}^\infty X^m\sum_{n=m+1}^\infty n^2\mathfrak Z_n(\zeta_n;m,3)Y^n\\
&=\frac{1}{X}\log\left(\frac{(1-Y)\bigl((1-Y)^3-(X Y+2 Y-3)X Y^2\bigr)}{\bigl((1-Y)^3-X Y^3\bigr)\bigl((1-Y)-X Y\bigr)}\right)\\
&=\frac{1}{X}\left(\sum_{m=1}^\infty\frac{1}{m}\left(\frac{X Y^3}{(1-Y)^3}\right)^m+\sum_{m=1}^\infty\frac{1}{m}\left(\frac{X Y}{1-Y}\right)^m\right.\\
&\left.-\sum_{m=1}^\infty\frac{1}{m}\left(\frac{(X Y+2 Y-3)X Y^2}{(1-Y)^3}\right)^m
\right)\,.
\end{align*} 
Now, the part of the first two summations is given by 
\begin{align*}
&\frac{1}{X}\left(\sum_{m=1}^\infty\frac{X^m Y^{3 m}}{m}\sum_{j=0}^\infty\binom{3 m+j-1}{3 m-1}Y^j+\sum_{m=1}^\infty\frac{X^m Y^{m}}{m}\sum_{j=0}^\infty\binom{m+j-1}{m-1}Y^j\right)\\
&=\sum_{m=0}^\infty\frac{X^m}{m+1}\sum_{j=0}^\infty\binom{3 m+j+2}{3 m+2}Y^{3 m+j+3}+\sum_{m=0}^\infty\frac{X^m}{m+1}\sum_{j=0}^\infty\binom{m+j}{m}Y^{m+j+1}\\
&=\sum_{m=0}^\infty\frac{X^m}{m+1}\sum_{n=3 m+3}^\infty\binom{n-1}{3 m+2}Y^n+\sum_{m=0}^\infty\frac{X^m}{m+1}\sum_{n=m+1}^\infty\binom{n-1}{m}Y^n\,. 
\end{align*} 
The part of the last summation is given by 
\begin{align*}
&\frac{1}{X}\sum_{m=1}^\infty\frac{(X Y+2 Y-3)^m X^m Y^{2 m}}{m}\sum_{j=0}^\infty\binom{3 m+j-1}{3 m-1}Y^j\\
&=\sum_{\mu=0}^\infty\frac{1}{\mu+1}\sum_{k=0}^{\mu+1}\binom{\mu+1}{k}X^k Y^k(2 Y-3)^{\mu+1-k}X^{\mu}\sum_{j=0}^\infty\binom{3\mu+j+2}{3\mu+2}Y^{2\mu+j+2}\\
&=\sum_{m=0}^\infty X^m\sum_{k=0}^{\fl{\frac{m+1}{2}}}\frac{1}{m-k+1}\binom{m-k+1}{k}Y^k\sum_{i=0}^{m-2 k+1}\binom{m-2 k+1}{i}2^i Y^i(-3)^{m-2 k+1-i}\\
&\qquad\qquad\times\sum_{j=0}^\infty\binom{3 m-3 k+j+2}{3 m-3 k+2}Y^{2 m-2 k+j+2}\\ 
&=\sum_{m=0}^\infty X^m\sum_{k=0}^{\fl{\frac{m+1}{2}}}\sum_{i=0}^{m-2 k+1}\frac{1}{m-k+1}\binom{m-k+1}{k}\binom{m-2 k+1}{i}2^i(-3)^{m-2 k+1-i}\\
&\qquad\qquad\times\sum_{n=2 m-k+i+2}^\infty\binom{n+m-2 k-i}{3 m-3 k+2}Y^n\,.
\end{align*} 
Combining two parts together, we obtain the desired result.  
\end{proof}

\section{Further works}  

By a similar technique as in the proof of Theorem \ref{th:s2}, it seems possible to obtaint an explicit expression of $\mathfrak Z_n(\zeta_n;m,4)$ too. In fact, by Theorem \ref{th:zz-det} with Corollary \ref{cor:zz-det2}, we can get 
\begin{align*}
\mathfrak Z_n(\zeta_n;1,4)&=\frac{(n-1)(n^3+n^2-109 n+251)}{6!}\,,\\ 
\mathfrak Z_n(\zeta_n;2,4)&=\frac{2(n-1)(n-2)(n^6+3 n^5-148 n^4+810 n^3+12869 n^2-101613 n+188878)}{10!}\,.\\ 
\end{align*}
However, the calculation to obtain a general formula is more difficult than one might imagine.

\section*{Competing Interests}  
The author declares no competing interest.



\end{document}